\documentclass[12pt]{article}
\usepackage{amssymb}
\usepackage{amsmath}
\usepackage{subfig}
\usepackage{pgfplots}
\newtheorem{proposition}{Proposition}

\newtheorem{lemma}{Lemma}
\newtheorem{remark}{Remark}

\newenvironment{proof}{\noindent{\bf Proof: }}{\hfill\fbox{}\vspace*{3mm}}

\def\b{\bf}

\def\bu{\vrule height .9ex width .8ex depth -.1ex \ }

\def\ga{\gamma}

\def\la{\lambda}

\def\th{\theta}
\def\dfrac{\displaystyle \frac}
\def\dsum{\displaystyle \sum}

\def\b0{{\bf{0}}}
\def\ba{{\bf{a}}}
\def\bb{{\bf{b}}}
\def\bh{{\bf{h}}}
\def\bu{{\bf{u}}}
\def\bx{{\bf{x}}}
\def\by{{\bf{y}}}
\def\bA{{A}}
\def\bD{{D}}
\def\bE{{E}}
\def\bI{{I}}
\def\={\!\!\!=\!\!\!}

\def\oth{\overline{\theta}}

\begin{document}
\title{\textbf{A continuous time multi-echelon inventory model for deteriorating items with transshipment}}
\author{}
\date{}
\maketitle
\begin{abstract}
In this paper, we propose a continuous time model for a multi-echelon inventory system. 
Items in the inventory are deteriorating.
Lateral transshipment are allowed among the warehouses in the same echelon where transshipment rate depends on
the inventory level of the corresponding warehouses. 
We give some sufficient conditions for the equilibrium points of the system to be stable.
By aggregating the warehouses in an echelon, a fast procedure is developed for finding the equilibrium values of inventory level at each warehouse.

\end{abstract}

\section{Introduction}

In a multi-echelon inventory system, inventory may be shared among warehouses in the same echelon in order to prevent shortage. Such transportation of inventory is called {\it lateral transshipment}. 
By lateral transshipment, operating cost of the inventory system can be reduced \cite{Robinson}.
A number of research works have been published in this area.
Diks and de Kok \cite{Diks} considered a two echelon network with transshipments and proposed an optimal rebalancing policy at the retailer level to maintain all inventory at each retailer at a balanced position.
Tagaras \cite{Tagaras} investigated the effects of pooling the inventory on minimizing the operation costs as well as optimizing the service levels.
Hochmuth and K\"{o}chel \cite{Hochmuth} used simulation approach to determine the order size and transshipment levels in a multi-location inventory system. 

In what follows, we present a brief review on lateral transshipment based on the review paper by Paterson et al. \cite{Paterson}.
There are two main strands of lateral transshipment in literature: proactive transshipment and reactive transshipment.
In proactive transshipment models, stock is redistributed through lateral transshipment among all stocking points in an echelon at predetermined moments in time. 
This policy is useful for high transshipment handling costs since transshipments are arranged in advance. 
While in reactive transshipment models, transshipment occurs from one stock point that has sufficient stock to a stock point that faces a stock-out (or at risk of stock-out).
This policy is suitable when the transshipment costs are low comparing to the holding costs.
As pointed out in \cite{Paterson}, there is no research work considering continuous time model on proactive lateral transshipment.
In this paper, we propose a continuous time model for a multi-echelon inventory system in which redistribution of inventory among warehouses in the same echelon takes place continuously.
We also integrate reactive lateral transshipment into the proposed model. 
The transshipment rate between any two warehouses in the same echelon is set to be depended on the inventory levels at those two warehouses.
This policy helps to reduce the shortage situation.

Product deterioration occurs in many inventory system such as medicine, food and electronic products.
Deteriorating inventory models have been studied widely in the past few decades.
Recently, there are some research works focus on inventory model for deteriorating items in multi-echelon supply chain.
Rau et al. \cite{Rau} proposed a model for deteriorating inventory on optimizing the joint total cost among the supplier, producer and buyer. 
Wang et al. \cite{Wang} proposed a coordination mechanism to determine the timing and quantities of deliveries
in cooperation with up-/down-stream members in the supply chain.
In this paper, we assume that the amount of items deteriorated depends on the current inventory level at the warehouses.
Such setting in commonly adopted in models for deteriorating items, for example \cite{Mak}.
We also assume that the quantities of deliveries from upstream to downstream depend on the inventory levels of the two echelons in the product transportation.

The remainder of the paper is organized as follows. In Section 2, we present one basic model and one aggregated model for the warehouses in one particular echelon of the system.
Three special cases are discussed and numerical examples are given.
In Section 3, we present a model for a multi-echelon inventory system, in which all warehouses in the same echelon are considered as one aggregated warehouse.
We develop a procedure for finding the inventory level in equilibrium of a warehouse a particular echelon based on Newton's method. 
The paper is concluded in Section 4 to address further research issues. 

\section{The one-echelon case}
In this section, we first present a basic model for the warehouses in one particular echelon of the system.
The model is also applicable in modeling a multi-location inventory system with lateral transshipment.
Suppose that there are $n\ge 1$ warehouses in the echelon we are considering.
The following notations for each warehouse $i$ and $j$ $(1 \leq i,j \leq n)$ and time $t\geq 0$ are used in this section:

\begin{tabular}{ll}
$L_i$ & maximum inventory level\\
$y_i(t)$ & inventory level $(0 \leq y_i(t)\leq L_i)$\\
$\mu_i$ & maximum supply rate\\
$\th_i$ & percentage of items deteriorated per unit time\\
$\la_i$ & demand rate\\
$\ga_{ij}$ & maximum transshipment rate from warehouse $i$ to warehouse $j$ ($\ga_{ii}=0$)\\
\end{tabular}

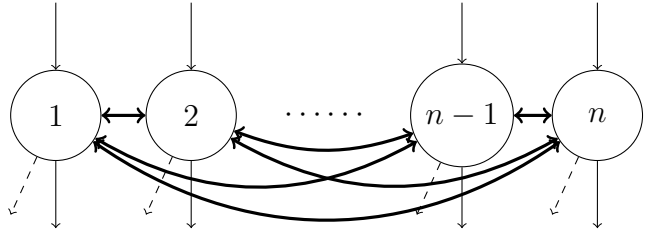
\begin{figure} 
\centering 
\begin{tikzpicture} [domain=0:20,scale=0.6] 

\node[minimum size=12mm,draw,circle] (y1) at (0,0) {$1$};
\node[minimum size=12mm,draw,circle] (y2) at (3,0) {$2$};
\node at (6,0) {$\cdots \cdots$};
\node[minimum size=12mm,draw,circle] (y3) at (9,0) {$n-1$};
\node[minimum size=12mm,draw,circle] (y4) at (12,0) {$n$};

\draw[<->,very thick] (y1) -- (y2);
\draw[<->,very thick] (y3) -- (y4);
\draw[<->,very thick] (y1) edge [bend right=30] (y3);
\draw[<->,very thick] (y1) edge [bend right=35] (y4);
\draw[<->,very thick] (y2) edge [bend right=20] (y3);
\draw[<->,very thick] (y2) edge [bend right=30] (y4);

\draw[->] (y1) -- (0, -2.5);
\draw[->] (y2) -- (3, -2.5);
\draw[->] (y3) -- (9, -2.5);
\draw[->] (y4) -- (12, -2.5);

\draw[->] (0, 2.5) -- (y1);
\draw[->] (3, 2.5) -- (y2);
\draw[->] (9, 2.5) -- (y3);
\draw[->] (12, 2.5) -- (y4);

\draw[dashed,->] (y1) -- (-1, -2.2);
\draw[dashed,->] (y2) -- (2, -2.2);
\draw[dashed,->] (y3) -- (8, -2.2);
\draw[dashed,->] (y4) -- (11, -2.2);

\end{tikzpicture} 
\caption{An echelon of the inventory system.}
\end{figure} 

\bigskip
An echelon of the system is graphically illustrated in Figure 1.
The solid line arrows represent the supply and demand at the warehouses and the dashed line arrows represent the deteriorated items screened out from the warehouses.
Each bold arrow represents a possible route of transshipment.  
The supply rate to the warehouse $i$ depends on the inventory level.
If the inventory level is close to the maximum level $L_i$ then the supply rate should be low.
On the other hand, if the inventory level is close to zero then the supply rate should be high.
Hence, the supply rate is given by
$$
\mu_i\dfrac{(L_i-y_i(t))}{L_i}.
$$
Similarly, the lateral transshipment rate from warehouse $i$ to warehouse $j$ depends on its inventory levels at $i$ and $j$.
The lateral transshipment rate should be high if the inventory level at $i$ is high and the inventory level at $j$ is low.
Hence, transshipment rate from $i$ to $j$ is given by
$$
\ga_{ij}\dfrac{y_i(t)}{L_i}\dfrac{(L_j-y_j(t))}{L_j}.
$$
The deterioration rate is proportional to the inventory level on hand, i.e.
$$\th_i y_i(t).$$
In this paper, we assume that the items are screened out immediately from the inventory once the items deteriorate.

\subsection{The basic model}
For simplicity of discussion, we assume that the maximum transshipment rate from warehouse $i$ to warehouse $j$ is the same as that from warehouse $j$ to warehouse $i$, i.e. $\ga_{ij}=\ga_{ji}$ $(i,j=1,\cdots,n)$.
The analysis can be easily extended to the case where this assumption does not hold.
The rate of change of the inventory level at warehouse $i$ at time $t$ is
the sum of the supply rate and the total transshipment rate from other warehouses to $i$, subtracting the 
demand rate, the deterioration rate and the total transshipment rate from $i$ to other warehouses:
$$
\begin{array}{rcl}
\dfrac{dy_i(t)}{dt} &=& \mu_i \dfrac{(L_i-y_i(t))}{L_i}+\dsum_{j=1}^n \ga_{ij}\dfrac{y_j(t)}{L_j}\dfrac{(L_i-y_i(t))}{L_i}  -\la_i - \th_i y_i(t)
\\
&&-\dsum_{j=1}^n \ga_{ij}\dfrac{y_i(t)}{L_i}\dfrac{(L_j-y_j(t))}{L_j}\\
&=&(\mu_i-\la_i)-\Big( \dfrac{\mu_i}{L_i}+\th_i+ \dsum_{j=1}^n \dfrac{\ga_{ij}}{L_i}\Big)y_i(t)+\dsum_{j=1}^n \dfrac{\ga_{ij}}{L_j}y_j(t).
\end{array}
$$

Define the inventory level vector at time $t$ by $\by(t)=(y_1(t),y_2(t),\cdots,y_n(t))^t$.
Rewriting the system of $n$ linear differential equations in matrix form we have:
\begin{equation}\label{nlwh}
\by'=\bA \by+\bb,
\end{equation}
where
$$
\bA=
\begin{pmatrix}
-\Big( \dfrac{\mu_1}{L_1}+\th_1+ \dsum_{j=1}^n \dfrac{\ga_{1j}}{L_1}\Big) & \dfrac{\ga_{12}}{L_2} & \cdots & \dfrac{\ga_{1n}}{L_n} \cr
\dfrac{\ga_{21}}{L_1} & \ddots & & \vdots \cr
\vdots & & \ddots & \vdots \cr
\dfrac{\ga_{n1}}{L_1} & \cdots & & -\Big( \dfrac{\mu_n}{L_n}+\th_n+ \dsum_{j=1}^n \dfrac{\ga_{nj}}{L_n}\Big)\cr
\end{pmatrix}
\mbox{ and }
\bb=
\begin{pmatrix}
\mu_1-\la_1 \cr
\mu_2-\la_2 \cr
\vdots \cr
\mu_n-\la_n
\end{pmatrix}.
$$
The following lemma is needed in the proofs of several propositions.
\begin{lemma}
The real part of each of the eigenvalues of $\bA$ is negative.
\end{lemma}
\begin{proof}
By applying the Gershgorin Circle Theorem \cite[p. 357]{Golub} to $\bA^t$, the real part of each of its eigenvalues lies in the interval
$$
\bigcup_{i=1}^n \left( -\Big( \dfrac{\mu_i}{L_i}+\th_i+ 2\dsum_{j=1}^n \dfrac{\ga_{ij}}{L_i}\Big), 
-\Big( \dfrac{\mu_i}{L_i}+\th_i\Big)\right) \subset (-\infty,0).
$$ 
Therefore, the real part of each of the eigenvalues of $\bA$ is negative. 
\end{proof}

From Lemma 1, it is easy to see that $A$ is non-singular.
Hence, the solution of system (\ref{nlwh}) can be obtained by the following proposition.
\begin{proposition}
The solution of system (\ref{nlwh}) is
$$
\by(t)=e^{\bA t} \by(0)+\bA^{-1}(e^{\bA t}-I) \bb,
$$
where $e^{\bA t}=\sum_{n=0}^\infty \bA^n t^n/n!$.
Furthermore, if $y_i(0)>0$ and $\mu_i>\la_i$, $(i=1,\cdots,n)$,
then $\by(t)>0$ for all $t>0$ and $\by^*>0.$
\end{proposition}
\begin{proof}
See \cite{Arrow}.
\end{proof}

We remark that when $n$ is large, the exponential of $\bA$ is difficult to compute in general.
Another quantity of interest is the equilibrium value of inventory level at each warehouse.
In equilibrium, we have $\b0=\bA \by^*+\bb$, i.e.,
$$
\by^*=-\bA^{-1} \bb.
$$
To classify the equilibrium point, one may follow the analysis in \cite{HubbardWest}, \cite[p. 261]{Sachdev}:
\begin{proposition}
The equilibrium point $\by^*$ is a stable one.
\end{proposition}
The proof follows from the fact that, by Lemma 1, the real part of the eigenvalues of $\bA$ is negative \cite{HubbardWest}.

\subsection{Special cases}
In this subsection, we give three special cases of the above model.

\noindent(1) \underline{$n=2:$} Suppose that there are two warehouses in the echelon. 
For simplicity, denote $\ga=\ga_{12}=\ga_{21}$. The matrix $\bA$ is now
	$$\bA=
	\begin{pmatrix}
-\Big( \dfrac{\mu_1}{L_1}+\th_1+ \dfrac{\ga}{L_1}\Big) & \dfrac{\ga}{L_2} \cr
\dfrac{\ga}{L_1} & -\Big( \dfrac{\mu_2}{L_2}+\th_2+  \dfrac{\ga}{L_2}\Big)\cr
\end{pmatrix}.
	$$
	The determinant of $\bA$ is given by
	$$
	\det(\bA)=\Big( \dfrac{\mu_1}{L_1}+\th_1\Big)\Big( \dfrac{\mu_2}{L_2}+\th_2\Big)+\dfrac{\ga}{L_1}\Big( \dfrac{\mu_2}{L_2}+\th_2\Big)+\dfrac{\ga}{L_2}\Big( \dfrac{\mu_1}{L_1}+\th_1\Big)>0,
	$$
	and the exponential of $\bA t$ is given by
	$$
	e^{\bA t}=e^{\eta_1 t} \bI+ \dfrac{e^{\eta_1 t}-e^{\eta_2 t}}{\eta_1-\eta_2}(\bA-\eta_1 \bI),
	$$
		where 
$$\eta_1=\dfrac{-T+\sqrt{T^2-4\det(\bA)}}{2}, \eta_2=\dfrac{-T-\sqrt{T^2-4\det(\bA)}}{2} \mbox{ and } T=\dsum_{i=1}^2	\Big(\dfrac{\mu_i}{L_i}+\th_i+\dfrac{\ga}{L_i}\Big).
$$

Hence, the solution of the system is
	$$
	\by(t)=e^{\bA t}\by(0)
		+\dfrac{1}{\det(\bA)}\begin{pmatrix}
 \dfrac{\mu_2}{L_2}+\th_2+ \dfrac{\ga}{L_2} & \dfrac{\ga}{L_2} \cr
\dfrac{\ga}{L_1} &  \dfrac{\mu_1}{L_1}+\th_1+  \dfrac{\ga}{L_1}\cr
\end{pmatrix}
(I-e^{\bA t})
\begin{pmatrix}
\mu_1-\la_1 \cr
\mu_2-\la_2\cr
\end{pmatrix},
	$$
and the equilibrium point is
$$
\by^*=\dfrac{1}{\det(\bA)}\begin{pmatrix}
 \dfrac{\mu_2}{L_2}+\th_2+ \dfrac{\ga}{L_2} & \dfrac{\ga}{L_2} \cr
\dfrac{\ga}{L_1} &  \dfrac{\mu_1}{L_1}+\th_1+  \dfrac{\ga}{L_1}\cr
\end{pmatrix}
\begin{pmatrix}
\mu_1-\la_1 \cr
\mu_2-\la_2\cr
\end{pmatrix}.
$$

\noindent(2) \underline{A star network:} Suppose that the warehouses form a star network, see Figure 2.
(For simplicity, we only show the transshipment arrows in the figure and skip the other arrows.)
Then the matrix $\bA$ is given by $\bA=\bD+\bE$ where $\bD$ is a diagonal matrix:
$$
D_{ii}=\left\{
\begin{array}{ll}
-\Big( \dfrac{\mu_1}{L_1}+\th_1+ \dsum_{j=2}^n \dfrac{\ga_{1j}}{L_1}\Big), & \mbox{ if }i=1;\\
-\Big( \dfrac{\mu_i}{L_i}+\th_i+ \dfrac{\ga_{i1}}{L_1}\Big), & \mbox{ if }i=2,\cdots,n,
\end{array}
\right.
$$
and $\bE$ is a rank two matrix:
$$
\bE=
\begin{pmatrix}
0 & \dfrac{\ga_{12}}{L_2} & \cdots & \dfrac{\ga_{1n}}{L_n} \cr
\dfrac{\ga_{21}}{L_1} & & &  \cr
\vdots & & \ddots & \cr
\dfrac{\ga_{n1}}{L_1} &  & & 0\cr
\end{pmatrix}
=\begin{pmatrix}
0 & 1 \cr
\ga_{21}& 0 \cr
\vdots & \vdots \cr
\ga_{n1}& 0\cr
\end{pmatrix}
\begin{pmatrix}
\dfrac{1}{L_1} & 0 & \cdots & 0\cr
0 & \dfrac{\ga_{12}}{L_2} & \cdots & \dfrac{\ga_{1n}}{L_n} \cr
\end{pmatrix}.
$$
The inverse of $\bA$ can be computed by the Sherman-Morrison-Woodbury formula.

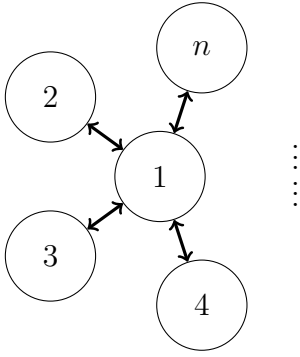
\begin{figure} 
\centering 
\begin{tikzpicture} [domain=0:20,scale=0.6] 

\node[minimum size=12mm,draw,circle] (y1) at (0:0) {$1$};
\node[anchor=center] at (3,-.2) {$\vdots$};
\node[anchor=center] at (3,.6) {$\vdots$};
\node[minimum size=12mm,draw,circle] (yn) at (72:3) {$n$};
\node[minimum size=12mm,draw,circle] (y2) at (72*2:3) {$2$};
\node[minimum size=12mm,draw,circle] (y3) at (72*3:3) {$3$};
\node[minimum size=12mm,draw,circle] (y4) at (72*4:3) {$4$};

\draw[<->,very thick] (y1) -- (y2);
\draw[<->,very thick] (y1) -- (y3);
\draw[<->,very thick] (y1) -- (y4);
\draw[<->,very thick] (y1) -- (yn);

%

\end{tikzpicture} 
\caption{A star network of warehouses.}
\end{figure} 

\begin{proposition}
{\rm (Sherman-Morrison-Woodbury Formula \cite[p. 65]{Golub})} Let $M$ be a non-singular $r \times r$ matrix, $u$ and $v$
be two $r \times l$ $(l \leq r)$ matrices such that the matrix $(I_l + v^tM^{-1}u)$ is non-singular. Then we have:
$$(M+uv^t)^{-1}=M^{-1}-M^{-1}u(I_l + v^tM^{-1}u)^{-1}v^tM^{-1}.$$
\end{proposition}
Then by Proposition 3, we have
$$
(D+E)^{-1}=D^{-1}-D^{-1}\left(\begin{smallmatrix}
0 & 1 \cr
\ga_{21}& 0 \cr
\vdots & \vdots \cr
\ga_{n1}& 0\cr
\end{smallmatrix}\right)
\left(
I_2+
\left(\begin{smallmatrix}
\frac{1}{L_1} & 0 & \cdots & 0\cr
0 & \frac{\ga_{12}}{L_2} & \cdots & \frac{\ga_{1n}}{L_n} \cr
\end{smallmatrix}\right)D^{-1}
\left(\begin{smallmatrix}
0 & 1 \cr
\ga_{21}& 0 \cr
\vdots & \vdots \cr
\ga_{n1}& 0\cr
\end{smallmatrix}\right)
\right)^{-1}
\left(\begin{smallmatrix}
\frac{1}{L_1} & 0 & \cdots & 0\cr
0 & \frac{\ga_{12}}{L_2} & \cdots & \frac{\ga_{1n}}{L_n} \cr
\end{smallmatrix}\right)D^{-1},
$$
or
$$
(D+E)^{-1}=D^{-1}-
\Big(\dsum_{i=1}^n\frac{\ga_{1i}^2}{L_i D_{ii}}\Big)^{-1}
\left(\begin{smallmatrix}
0 & \frac{1}{D_{11}} \cr
\frac{\ga_{21}}{D_{22}}& 0 \cr
\vdots & \vdots \cr
\frac{\ga_{n1}}{D_{nn}}& 0\cr
\end{smallmatrix}\right)
\left(\begin{smallmatrix}
1&\frac{1}{L_1D_{11}} \cr
\sum_{i=2}^n\frac{\ga_{1i}^2}{L_i D_{ii}} & 1 \cr
\end{smallmatrix}\right)
\left(\begin{smallmatrix}
\frac{1}{L_1D_{11}} & 0 & \cdots & 0\cr
0 & \frac{\ga_{12}}{L_2D_{22}} & \cdots & \frac{\ga_{1n}}{L_nD_{nn}} \cr
\end{smallmatrix}\right).
$$

\noindent(3) \underline{A linear network:} Suppose that the warehouses form a linear network, see Figure 3.
Then the matrix $\bA$ is given by
$$
\bA=
\begin{pmatrix}
-\Big( \dfrac{\mu_1}{L_1}+\th_1+  \dfrac{\ga_{12}}{L_1}\Big) & \dfrac{\ga_{12}}{L_2} &   & & 0 \cr
\dfrac{\ga_{21}}{L_1} & -\Big( \dfrac{\mu_2}{L_2}+\th_2+  \dfrac{\ga_{21}}{L_2}+\dfrac{\ga_{23}}{L_2}\Big) &\dfrac{\ga_{23}}{L_3} &  & \cr
& \ddots & \ddots & \ddots & \cr
&&\ddots&\ddots&\dfrac{\ga_{n-1,n}}{L_n}\cr
0 &  & &\dfrac{\ga_{n,n-1}}{L_{n-1}} & -\Big( \dfrac{\mu_n}{L_n}+\th_n+  \dfrac{\ga_{n,n-1}}{L_n}\Big)\cr
\end{pmatrix},
$$
which is a tridiagonal matrix. By Proposition 1, the solution to the system is 
$$
\by(t)=e^{\bA t} \by(0)+\bA^{-1}e^{\bA t}\bb-\bA^{-1} \bb,
$$
and the equilibrium point is 
$$
\by^*=-\bA^{-1}\bb.
$$
The inverse of $\bA$ can be computed by a simple algorithm presented in \cite{Lewis}.
If $L_1=L_2=\cdots=L_n$, then $\bA$ is symmetric.
There are fast algorithms for finding the exponential of a symmetric tridiagonal matrix, see for instance \cite{Lu}.

\begin{figure} 
\centering 
\begin{tikzpicture} [domain=0:20,scale=0.6] 

\node[minimum size=12mm,draw,circle] (y1) at (0,0) {$1$};
\node[minimum size=12mm,draw,circle] (y2) at (3,0) {$2$};
\node at (6,0) {$\cdots \cdots$};
\node[minimum size=12mm,draw,circle] (y3) at (9,0) {$n-1$};
\node[minimum size=12mm,draw,circle] (y4) at (12,0) {$n$};

\draw[<->,very thick] (y1) -- (y2);
\draw[<->,very thick] (y3) -- (y4);
\draw[<->,very thick] (y2) -- (5,0);
\draw[<->,very thick] (y3) -- (7,0);

\draw[->] (y1) -- (0, -2);
\draw[->] (y2) -- (3, -2);
\draw[->] (y3) -- (9, -2);
\draw[->] (y4) -- (12, -2);

\draw[->] (0, 2) -- (y1);
\draw[->] (3, 2) -- (y2);
\draw[->] (9, 2) -- (y3);
\draw[->] (12, 2) -- (y4);

\draw[dashed,->] (y1) -- (-1, -1.8);
\draw[dashed,->] (y2) -- (2, -1.8);
\draw[dashed,->] (y3) -- (8, -1.8);
\draw[dashed,->] (y4) -- (11, -1.8);

\end{tikzpicture} 
\caption{A linear network of warehouses.}
\end{figure}
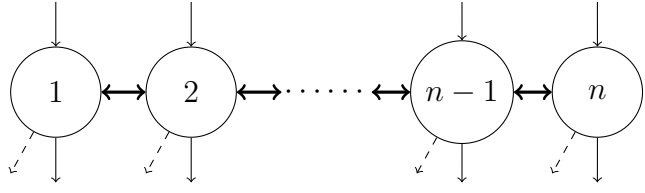 

\subsection{Aggregated model for the warehouses}
In this section, we propose an aggregated inventory model for the warehouses by aggregating the supply, demand and inventory at the warehouses.
For the deterioration process, we take the average deterioration percentage.
Let
$$
L_a=\sum_{i=1}^n L_i, \quad \mu_a=\sum_{i=1}^n \mu_i, \quad \la_a=\sum_{i=1}^n \la_i \quad \mbox{and} \quad
\oth=\dfrac{1}{n}\sum_{i=1}^n \th_i.
$$
The rate of change of the aggregated inventory level at time $t$ is the total supply rate subtracting the total demand rate and the total deterioration rate:
\begin{equation}
\label{alwh}
\dfrac{dy_a(t)}{dt}=\mu_a\dfrac{L_a-y_a(t)}{L}-\la_a-\oth y_a(t)=-\Big(\dfrac{\mu_a}{L_a}+\oth\Big)y_a(t)+(\mu_a-\la_a).
\end{equation}
The solution of system (\ref{alwh}) is
$$
y_a(t)=\exp\big[-(\mu_a/L_a+\oth)t\big]y_a(0)-\dfrac{\exp\big[-(\mu_a/L_a+\oth)t\big]}{\mu_a/L_a+\oth}(\mu_a-\la_a)+\dfrac{\mu_a-\la_a}{\mu_a/L+\oth},
$$
and the equilibrium value of the aggregated inventory level is
$$
y_a^*=\lim_{t\to \infty}y_a(t) = \dfrac{\mu_a-\la_a}{\mu_a/L_a+\oth}.
$$

The aggregated model is a simplified version of the model in Section 2.1.
The following propositions compare the results of the two models.

\begin{proposition}
Let $(y_1(t),\cdots,y_n(t))$ and $y_a(t)$ be the solutions of systems (\ref{nlwh}) and (\ref{alwh}) respectively.
If $$
\dfrac{\mu_1}{L_1}=\dfrac{\mu_2}{L_2}=\cdots=\dfrac{\mu_n}{L_n} \quad \mbox{and} \quad \th_1=\cdots=\th_n$$
then
$$\dsum_{i=1}^n y_i^* =y_a^*.$$
Furthermore, if $\sum_{i=1}^n y_i(0) =y_a(0)$, then
$$\dsum_{i=1}^n y_i(t) =y_a(t),\quad \mbox{for } t \geq 0.$$
\end{proposition}

\begin{proof}
We prove the second part of the proposition. The first part can be proved similarly.
The first two conditions imply
$$
\dfrac{\mu_1}{L_1}=\dfrac{\mu_2}{L_2}=\cdots=\dfrac{\mu_n}{L_n}=\dfrac{\mu_a}{L_a}\quad \mbox{and} \quad\th_1=\cdots=\th_n =\oth.
$$
The solution of system (\ref{nlwh}) satisfies
$$
\by'=A\by+\bb,
$$
which means
$$
\begin{array}{rcl}
(1,\cdots,1)\by'&=&(1,\cdots,1)A\by+(1,\cdots,1)\bb\\
\dsum_{i=1}^n y_i'(t)&=&\left( -\Big(\dfrac{\mu_1}{L_1}+\th_1\Big),\cdots,-\Big(\dfrac{\mu_n}{L_n}+\th_n\Big) \right)\by+(\mu_a-\la_a)\\
\dsum_{i=1}^n y_i'(t)&=& -\Big(\dfrac{\mu_a}{L_a}+\oth\Big)\dsum_{i=1}^n y_i(t)+(\mu_a-\la_a).\\
\end{array}
$$
Hence $\sum_{i=1}^n y_i(t)$ is a solution of system (\ref{alwh}). The proof follows from the uniqueness of the solution of (\ref{alwh}).
\end{proof}

\begin{proposition}
Let $(y_1^*,\cdots,y_n^*)$ and $y_a^*$ be the equilibrium points of systems (\ref{nlwh}) and (\ref{alwh}) respectively.
If $L_1=\cdots=L_n$, then
$$
\Big|y_a^*-\dsum_{i=1}^n y_i^*\Big| \leq \dfrac{\sqrt{n}(\mu_a-\la_a)}{\displaystyle\min_i(\mu_i/L_i+\th_i)}+\dfrac{\mu_a-\la_a}{\mu_a/L_a+\oth}.
$$
\end{proposition}

\begin{proof}
If $L_1=\cdots=L_n$, then the matrix $A$ is symmetric.
Therefore, applying the Euclidean norm on $A^{-1}$ yields
$$\|A^{-1}\|_2=\max_i|\psi_i(A^{-1})|,$$
where $\psi_i(A)$ is the $i$th eigenvalue of $A$. By Lemma 1, for $i=1,\cdots,n$,
$$
\psi_i(A) \in \left( -\max_i\Big( \dfrac{\mu_i}{L_i}+\th_i+ 2\dsum_{j=1}^n \dfrac{\ga_{ij}}{L_i}\Big), 
-\min_i\Big( \dfrac{\mu_i}{L_i}+\th_i\Big)\right).
$$
Hence, 
$$
\psi_i(A^{-1}) \in \left( -\Big[\min_i\Big( \dfrac{\mu_i}{L_i}+\th_i\Big)\Big]^{-1},
-\Big[\max_i\Big( \dfrac{\mu_i}{L_i}+\th_i+ 2\dsum_{j=1}^n \dfrac{\ga_{ij}}{L_i}\Big)\Big]^{-1}\right) \subset(-\infty,0),
$$
which gives
$$
\|A^{-1}\|_2=\max_i|\psi_i(A^{-1})| \leq \Big[\min_i\Big( \dfrac{\mu_i}{L_i}+\th_i\Big)\Big]^{-1}.
$$
Now, 
$$
\begin{array}{rcl}
\Big|y_a^*-\dsum_{i=1}^n y_i^*\Big| &\leq& \Big|\dsum_{i=1}^n y_i^*\Big| + |y_a^*|\\
&=&|(1,\cdots,1)A^{-1}\bb|+ \dfrac{\mu_a-\la_a}{\mu_a/L_a+\oth}\\
&\leq&\|(1,\cdots,1)\|_2\times  \|A^{-1}\|_2\times \|\bb\|_2+ \dfrac{\mu_a-\la_a}{\mu_a/L_a+\oth}\\
&\leq&\sqrt{n}\times \Big[\min_i\Big( \dfrac{\mu_i}{L_i}+\th_i\Big)\Big]^{-1}\times (\mu_a-\la_a)+ \dfrac{\mu_a-\la_a}{\mu_a/L_a+\oth}\\
&=&\dfrac{\sqrt{n}(\mu_a-\la_a)}{\displaystyle\min_i(\mu_i/L_i+\th_i)}+\dfrac{\mu_a-\la_a}{\mu_a/L_a+\oth}.
\end{array}
$$
\end{proof}

\subsection{Numerical examples}
We first give an example of three warehouses, i.e. $n=3$.
Suppose that the parameters are
$$
\left\{
\begin{array}{l}
L_1=100, \quad L_2=200, \quad L_3=200, \quad \mu_1=3, \quad \mu_2=4, \quad \mu_3=5, \\
\th_1=0.1, \quad \th_2=0.2, \quad \th_3=0.3, \quad \la_1=1, \quad \la_2=2, \quad \la_3=3, \\
\big[\ga_{ij}\big]_{ij}=\begin{pmatrix}
0 & 0.5 & 0.2 \\[-2mm]
0.5 & 0 & 1\\[-2mm]
0.2 & 1 &0 \\
\end{pmatrix}, \quad  y_1(0)=50, \quad y_2(0)=100, \quad y_3(0)=150.
\end{array}
\right.
$$

\begin{table} {\scriptsize
\begin{tabular}{ccccccccccc}
\hline
$t$&	10	&	20	&	30	&	40	&	50	&	60	&	70	&	80	&	90	&	100	\\
\hline
$y_1(t)$&	24.312	&	17.277	&	15.445	&	14.974	&	14.854	&	14.824	&	14.816	&	14.814	&	14.813	&	14.813	\\
$y_2(t)$&	19.360	&	10.393	&	9.394	&	9.274	&	9.258	&	9.255	&	9.255	&	9.255	&	9.255	&	9.255	\\
$y_3(t)$&	11.908	&	6.533	&	6.291	&	6.274	&	6.272	&	6.272	&	6.272	&	6.272	&	6.272	&	6.272	\\
\hline
\end{tabular}
\caption{The inventory levels at different time $t$.}
}\end{table}

The inventory levels at the three warehouses at different time $t$ are given in Table 1.
The equilibrium point is
$$
(y_1^*,y_2^*,y_3^*)=(14.813, 9.255, 6.272).
$$
It can be observed that the inventory levels reach equilibrium when $t$ approaches $100$.
We next consider aggregating the three warehouses into one warehouse with 
$$
L_a=500, \quad \mu_a=12, \quad \oth=0.2, \quad \la_a=6, \quad y_a(0)=300.
$$
The equilibrium point is
$$
y_a^*= 26.786.
$$

We next present some numerical examples on the equilibrium points for different values of the system parameters.
In all the numerical tests, we assume each $L_i=200, \mu_a=48, \la_a=24$ and 
$$
\th_i=i \times 0.05 \quad \mbox{for } i=1,2,\cdots,8.
$$ 
In Tables 2-4, we assume $\ga_{ij}=1$ for $i\neq j$ and for each value of $\mu_i=16,20,24$,
we solve for $n=2,4,8$ and $\la_i=4,8,12$. 
The total inventory level of the warehouses in equilibrium for each case is also presented. 
We then compare the results with the corresponding equilibrium value of the aggregated inventory model, which are presented in bold font in the tables.
We observe that when $\mu_i$ increases, all the inventory levels at each warehouses in equilibrium are also increased. 
When $\la_i$ increases, all the inventory levels at each warehouses in equilibrium are decreased.
For each fixed $n$, the error due to aggregating the inventory levels is more sensitive to the changes in $\la_i$ than the changes in $\mu_i$. 

\begin{table} \centering{
\scriptsize
\begin{tabular}{lllllll}
\hline
&	$n=2$	&	&	$n=4$ &	&	$n=8$	\\
\hline
$\la_i=4$&(91.4	 	67.3)	&	158.7	{\bf 154.8}	&	(88.4	 	66.3	 	53.0	 	44.2)	&	251.9	{\bf 234.1}	&	(81.3	 	62.8	 	51.2	 	43.2	 	37.4	 	32.9	 	29.4	 	26.6)	&	364.8	{\bf 314.8}\\
$\la_i=8$&(60.9	 	44.9)	&	105.8	{\bf 103.2}	&	(58.9	 	44.2	 	35.4	 	29.5)	&	168.0	{\bf 156.1}	&	(54.2	 	41.9	 	34.1	 	28.8	 	24.9	 	21.9	 	19.6	 	17.7)	&	243.2		{\bf 209.8}\\
$\la_i=12$&(30.5	 	22.4)	&	52.9	{\bf 51.6}	&	(29.5	 	22.1	 	17.7	 	14.7)	&	84.0	{\bf 78.0}	&	(27.1	 	20.9	 	17.1	 	14.4	 	12.5	 	11.0	 	9.8	 	8.9)	&	121.6		{\bf 104.9}\\
\hline
\end{tabular}
\caption{The equilibrium points of the inventory level when $\mu_i=16$.}

\begin{tabular}{lllllll}
\hline
&	$n=2$	&	&	$n=4$ &	&	$n=8$	\\
\hline
$\la_i=4$&(105.8	 	80.6)	&	186.5	\textbf{182.9}	&	(103.0	 	79.6	 	64.9	 	54.7)	&	302.2	\textbf{284.4}	&	(96.0	 	76.0	 	62.9	 	53.6	 	46.8	 	41.5	 	37.2	 	33.8)	&	447.7	\textbf{393.8}\\
$\la_i=8$&(79.4	 	60.5)	  & 139.8	\textbf{137.1}	&	(77.3	 	59.7	 	48.6	 	41.0)	&	226.6	\textbf{213.3}	  &	(72.0	 	57.0	 	47.2	 	40.2	 	35.1	 	31.1	 	27.9	 	25.3)	&	335.8	\textbf{295.4}\\
$\la_i=12$&(52.9	 	40.3)	&	93.2	\textbf{91.4}	&	(51.5	 	39.8	 	32.4	 	27.4)	&	151.1	\textbf{142.2}	    &	(48.0	 	38.0	 	31.4	 	26.8	 	23.4	 	20.7	 	18.6	 	16.9)	&	223.9	\textbf{196.9}\\
\hline
\end{tabular}
\caption{The equilibrium points of the inventory level when $\mu_i=20$.}

\begin{tabular}{lllllll}
\hline
&	$n=2$	&	&	$n=4$ &	&	$n=8$	\\
\hline
$\la_i=4$&(116.9	 	91.5)	&	208.4	\textbf{205.1}	&	(114.3	 	90.5	 	74.9	 	63.9)	&	343.6	\textbf{326.5}&	(107.6	 	86.9	 	72.9	 	62.8	 	55.1	 	49.1	 	44.3	 	40.3)	&	519.0	\textbf{463.8}\\
$\la_i=8$&(93.5	 	73.2)	  &	166.7	\textbf{164.1}	&	(91.4	 	72.4	 	59.9	 	51.1)	&	274.8	\textbf{261.2}	&	(86.1	 	69.5	 	58.3	 	50.2	 	44.1	 	39.3	 	35.4	 	32.3)	  &	415.2	\textbf{371.0}\\
$\la_i=12$&(70.1	 	54.9)	&	125.0	\textbf{123.1}	&	(68.6	 	54.3	 	44.9	 	38.3)	&	206.1	\textbf{195.9}	&	(64.6	 	52.1	 	43.7	 	37.7	 	33.1	 	29.5	 	26.6	 	24.2)	  &	311.4	\textbf{278.3}\\
\hline
\end{tabular}
\caption{The equilibrium points of the inventory level when $\mu_i=24$.}
}\end{table}

For the numerical test in Table 5, we assume all $\mu_i=20$ and all $\la_i=12$. We calculate the equilibrium points for the cases when
$$\ga_{ij}=0.1, 0.5, 1, 2,\mbox{ and } 5$$
for $i\neq j$ and $n=2,4,8$.
The corresponding equilibrium values of the aggregated inventory model are presented in bold font in the table.
We observe that the inventory levels at each warehouses in equilibrium are reduced when $\ga_{ij}$ increases.
It shows that the transshipment policy is useful in reducing the inventory levels, which means a reduction of inventory costs.

\begin{table} {\scriptsize
\begin{tabular}{lllllll}
\hline
&	$n=2$	&	&	$n=4$ &	&	$n=8$	\\
\hline
$\ga_{ij}=0.1$&(53.3	 	40.0)	&	93.3	 	\textbf{91.4}	&	(53.1	 	40.0	 	32.0	 	26.7)	&	151.9	\textbf{142.2}	&	(52.7	 	39.8	 	31.9	 	26.7	 	22.9	 	20.1	 	17.9	 	16.1)	&	228.1	\textbf{196.9}\\
$\ga_{ij}=0.5$&(53.1	 	40.2)	&	93.3	 	\textbf{91.4}	&	(52.4	 	39.9	 	32.2	 	27.0)	&	151.5	\textbf{142.2}	&	(50.4	 	38.9	 	31.7	 	26.8	 	23.1	 	20.4	 	18.2	 	16.5)	&	226.0	\textbf{196.9}\\
$\ga_{ij}=1$&(52.9	 	40.3)	&	93.2	 	\textbf{91.4}	&	(51.5	 	39.8	 	32.4	 	27.4)	&	151.1	\textbf{142.2}	&	(48.0	 	38.0	 	31.4	 	26.8	 	23.4	 	20.7	 	18.6	 	16.9)	&	223.9	\textbf{196.9}\\
$\ga_{ij}=2$&(52.5	 	40.6)	&	93.1	 	\textbf{91.4}	&	(50.0	 	39.6	 	32.8	 	28.0)	&	150.3	\textbf{142.2}	&	(44.4	 	36.4	 	30.9	 	26.9	 	23.7	 	21.3	 	19.3	 	17.6)	&	220.4	\textbf{196.9}\\
$\ga_{ij}=5$&(51.6	 	41.3)	&	92.9	 	\textbf{91.4}	&	(46.9	 	39.1	 	33.5	 	29.3)	&	148.7	\textbf{142.2}	&	(38.1	 	33.4	 	29.7	 	26.7	 	24.3	 	22.3	 	20.5	 	19.1)	&	214.0	\textbf{196.9}\\
\hline
\end{tabular}
\caption{The equilibrium points of the inventory level for different values of $\ga_{ij}$.}
}\end{table}

\section{A Multi-echelon model} 
In this section, we present a model for a multi-echelon inventory system.
Suppose that there are $m$ echelons in the inventory system.
The warehouses in each echelon are aggregated as one warehouse by the method described in Section 2.3.
Therefore, an $m$ echelon inventory system can be modelled by an $m$ aggregated warehouses model, see Figure 4.
The following notations for each aggregated warehouse $i$ $(1 \leq i \leq m)$ and time $t \geq 0$ are used in this section:

\begin{tabular}{ll}
$C_i$ & maximum inventory level\\
$x_i(t)$ & inventory level $(0 \leq x_i(t)\leq C_i)$\\
$\mu^c_i$ & maximum supply rate\\
$\th^c_i$ & average percentage of items deteriorated per unit time\\
$\la^c$ & demand rate at the lowest echelon $m$\\
\end{tabular}

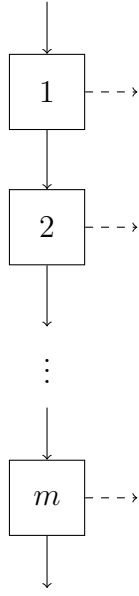
\begin{figure} 
\centering 
\begin{tikzpicture} [domain=0:20,scale=0.6] 

\node[minimum size=10mm,draw] (xm) at (0,0) {$m$};
\node at (0,3) {$\vdots$};
\node[minimum size=10mm,draw] (x2) at (0,6) {$2$};
\node[minimum size=10mm,draw] (x1) at (0,9) {$1$};

\draw[->] (0,2) -- (xm);
\draw[->] (x2) -- (0,3.8);
\draw[->] (x1) -- (x2);
\draw[->] (xm) -- (0,-2);
\draw[->] (0,11) -- (x1);

\draw[dashed,->] (xm) -- (2, 0);
\draw[dashed,->] (x2) -- (2, 6);
\draw[dashed,->] (x1) -- (2, 9);
%

\end{tikzpicture} 
\caption{A multi-echelon inventory system.}
\end{figure} 

\bigskip
With the above notations, the system of ordinary differential equations governing the inventory level in each echelon is given by:
$$\left\{
\begin{array}{rcl}
\dfrac{dx_1(t)}{dt}&=& \mu^c_1\dfrac{(C_1-x_1(t))}{C_1}-\th^c_1 x_1(t)-\mu^c_2\dfrac{x_1(t)}{C_1}\dfrac{(C_2-x_2(t))}{C_2};\\
\dfrac{dx_i(t)}{dt}&=& \mu^c_i\dfrac{x_{i-1}(t)}{C_{i-1}}\dfrac{(C_i-x_i(t))}{C_i}-\th^c_i x_i(t)-\mu^c_{i+1}\dfrac{x_i(t)}{C_i}\dfrac{(C_{i+1}-x_{i+1}(t))}{C_{i+1}};\\
\dfrac{dx_m(t)}{dt}&=& \mu^c_m\dfrac{x_{m-1}(t)}{C_{m-1}}\dfrac{(C_m-x_m(t))}{C_m}-\th^c_m x_m(t)-\la^c.\\
\end{array}
\right.
$$
Rearranging the terms we have:
$$\left\{
\begin{array}{rcl}
\dfrac{dx_1(t)}{dt}&=& \mu^c_1-\Big(\dfrac{\mu^c_1}{C_1}+\th^c_1+ \dfrac{\mu^c_2}{C_1}\Big)x_1(t)+\dfrac{\mu^c_2}{C_1C_2}x_1(t)x_2(t);\\
\dfrac{dx_i(t)}{dt}&=& \dfrac{\mu^c_i}{C_{i-1}}x_{i-1}(t) -\Big(\th^c_i + \dfrac{\mu^c_{i+1}}{C_i}\Big)x_i(t)
-\dfrac{\mu^c_i}{C_{i-1}C_i}x_{i-1}(t)x_i(t)
+\dfrac{\mu^c_{i+1}}{C_iC_{i+1}}x_i(t)x_{i+1}(t);\\
\dfrac{dx_m(t)}{dt}&=& -\la^c+ \dfrac{\mu^c_m}{C_{m-1}}x_{m-1}(t) -\th^c_m x_m(t)-\dfrac{\mu^c_m}{C_{m-1}C_m}x_{m-1}(t)x_m(t).\\
\end{array}
\right.
$$
In equilibrium, we have
$$\left\{
\begin{array}{lll}
0=F_1(x^*_1,\cdots,x^*_m)&=& \mu^c_1-\Big(\dfrac{\mu^c_1}{C_1}+\th^c_1+ \dfrac{\mu^c_2}{C_1}\Big)x^*_1+\dfrac{\mu^c_2}{C_1C_2}x^*_1x^*_2;\\
0=F_i(x^*_1,\cdots,x^*_m)&=& \dfrac{\mu^c_i}{C_{i-1}}x_{i-1}^* -\Big(\th^c_i + \dfrac{\mu^c_{i+1}}{C_i}\Big)x_i^*
-\dfrac{\mu^c_i}{C_{i-1}C_i}x_{i-1}^*x_i^* +\dfrac{\mu^c_{i+1}}{C_iC_{i+1}}x_i^*x_{i+1}^*;\\
0=F_m(x^*_1,\cdots,x^*_m)&=&-\la^c+ \dfrac{\mu^c_m}{C_{m-1}}x_{m-1}^* -\th^c_m x_m^*-\dfrac{\mu^c_m}{C_{m-1}C_m}x_{m-1}^*x_m^*.\\
\end{array}
\right.
$$

The following proposition gives a condition of obtaining a stable equilibrium point.
\begin{proposition}
Let $(x^*_1,\cdots,x^*_m)$ be the non-negative equilibrium point with $x^*_i\leq C_i$, $i=1,\cdots,m$.
If
\begin{equation}\label{stablecont}
\left\{
\begin{array}{rcl}
\dfrac{\mu^c_2}{C_1} & < & \dfrac{\mu^c_1}{C_1}+\th^c_1;\\
\dfrac{\mu^c_i}{C_{i-1}}+\dfrac{\mu^c_{i+1}}{C_{i+1}} & < & \th^c_i;\\
\dfrac{\mu^c_m}{C_{m-1}} & < & \th^c_m,\\
\end{array}
\right.
\end{equation}
then the equilibrium point is a stable one.
\end{proposition}
\begin{proof}
We consider the matrix:
$$
\left(
\begin{smallmatrix}
-\big(\frac{\mu^c_1}{C_1}+\th^c_1+ \frac{\mu^c_2}{C_1}\big)+\frac{\mu^c_2x^*_2 }{C_1C_2} & \frac{\mu^c_2x^*_1 }{C_1C_2}\cr
\frac{\mu^c_2}{C_1} -\frac{\mu^c_2x_2^* }{C_1C_2}& -\big(\th^c_2 + \frac{\mu^c_3}{C_2}\big)-\frac{\mu^c_2x_1^* }{C_1C_2} +\frac{\mu^c_3x_3^* }{C_2C_3} & \frac{\mu^c_3x_2^* }{C_2C_3} \cr
& \ddots & \ddots & \ddots\cr
&\frac{\mu^c_{m-1}}{C_{m-2}} -\frac{\mu^c_{m-1}x_{m-1}^* }{C_{m-2}C_{m-1}}& -\big(\th^c_{m-1} + \frac{\mu^c_m}{C_{m-1}}\big)-\frac{\mu^c_{m-1}x_{m-2}^* }{C_{m-2}C_{m-1}} +\frac{\mu^c_mx_m^* }{C_{m-1}C_m} & \frac{\mu^c_mx_{m-1}^* }{C_{m-1}C_m}\cr
&&\frac{\mu^c_m}{C_{m-1}}-\frac{\mu^c_mx_m^* }{C_{m-1}C_m} & -\th^c_m-\frac{\mu^c_mx_{m-1}^* }{C_{m-1}C_m}
\end{smallmatrix}\right).
$$
By applying the Gershgorin Circle Theorem \cite[p. 357]{Golub} to the matrix above, the real part of its eigenvalues are less than the maximum of
$$
\begin{array}{ll}
&\left\{
\begin{array}{l}
-\big(\dfrac{\mu^c_1}{C_1}+\th^c_1+ \dfrac{\mu^c_2}{C_1}\big)+\dfrac{\mu^c_2x^*_2 }{C_1C_2} + \dfrac{\mu^c_2x^*_1 }{C_1C_2}\\
-\big(\th^c_i + \dfrac{\mu^c_{i+1}}{C_i}\big)+\dfrac{\mu^c_i}{C_{i-1}}-\dfrac{\mu^c_ix_{i-1}^* }{C_{i-1}C_i} + \dfrac{\mu^c_{i+1}x_i^* }{C_iC_{i+1}} -\dfrac{\mu^c_ix_i^* }{C_{i-1}C_i} +\dfrac{\mu^c_{i+1}x_{i+1}^* }{C_iC_{i+1}}\\
-\th^c_m-\dfrac{\mu^c_mx_{m-1}^* }{C_{m-1}C_m} + \dfrac{\mu^c_m}{C_{m-1}}-\dfrac{\mu^c_mx_m^* }{C_{m-1}C_m}
\end{array}
\right.\\
\\
\leq &\left\{
\begin{array}{l}
-\big(\dfrac{\mu^c_1}{C_1}+\th^c_1+ \dfrac{\mu^c_2}{C_1}\big)+\dfrac{\mu^c_2 C_2 }{C_1C_2} + \dfrac{\mu^c_2 C_1 }{C_1C_2}\\
-\big(\th^c_i + \dfrac{\mu^c_{i+1}}{C_i}\big)+\dfrac{\mu^c_i}{C_{i-1}} + \dfrac{\mu^c_{i+1} C_i }{C_iC_{i+1}}  +\dfrac{\mu^c_{i+1} C_{i+1} }{C_iC_{i+1}}\\
-\th^c_m  + \dfrac{\mu^c_m}{C_{m-1}}
\end{array}
\right.\\
\\
\leq &\left\{
\begin{array}{l}
-\big(\dfrac{\mu^c_1}{C_1}+\th^c_1\big) + \dfrac{\mu^c_2 }{C_2}\\
-\th^c_i +\dfrac{\mu^c_i}{C_{i-1}} + \dfrac{\mu^c_{i+1}  }{C_{i+1}}  \\
-\th^c_m  + \dfrac{\mu^c_m}{C_{m-1}}
\end{array}
\right.
< \left\{
\begin{array}{l}
0.\\
0.\\
0.\\
\end{array}
\right.
\end{array}
$$
Therefore the equilibrium point is stable \cite{HubbardWest}.
\end{proof}

To solve for the equilibrium point, one may apply Newton's method \cite[p. 586]{Kincaid}. 
Let $(x^*_1(0),\cdots,x^*_m(0))$ be the initial guess, then the iterative scheme is
\footnotesize
$$
\begin{array}{ll}
&
\begin{pmatrix}
x^*_1(k+1)\cr
\vdots\cr
x^*_m(k+1)
\end{pmatrix}\\[5mm]
=&\begin{pmatrix}
x^*_1(k)\cr
\vdots\cr
x^*_m(k)
\end{pmatrix}\\[2mm]
&-
\left(
\begin{smallmatrix}
-\big(\frac{\mu^c_1}{C_1}+\th^c_1+ \frac{\mu^c_2}{C_1}\big)+\frac{\mu^c_2x^*_2(k)}{C_1C_2} & \frac{\mu^c_2x^*_1(k)}{C_1C_2}\cr
\frac{\mu^c_2}{C_1} -\frac{\mu^c_2x_2^*(k)}{C_1C_2}& -\big(\th^c_2 + \frac{\mu^c_3}{C_2}\big)-\frac{\mu^c_2x_1^*(k)}{C_1C_2} +\frac{\mu^c_3x_3^*(k)}{C_2C_3} & \frac{\mu^c_3x_2^*(k)}{C_2C_3} \cr
& \ddots & \ddots & \ddots\cr
&\frac{\mu^c_{m-1}}{C_{m-2}} -\frac{\mu^c_{m-1}x_{m-1}^*(k)}{C_{m-2}C_{m-1}}& -\big(\th^c_{m-1} + \frac{\mu^c_m}{C_{m-1}}\big)-\frac{\mu^c_{m-1}x_{m-2}^*(k)}{C_{m-2}C_{m-1}} +\frac{\mu^c_mx_m^*(k)}{C_{m-1}C_m} & \frac{\mu^c_mx_{m-1}^*(k)}{C_{m-1}C_m}\cr
&&\frac{\mu^c_m}{C_{m-1}}-\frac{\mu^c_mx_m^*(k)}{C_{m-1}C_m} & -\th^c_m-\frac{\mu^c_mx_{m-1}^*(k)}{C_{m-1}C_m}
\end{smallmatrix}\right)^{-1}\\
&\quad \times \begin{pmatrix}
F_1(x^*_1(k),\cdots,x^*_m(k))\cr
\vdots\cr
F_m(x^*_1(k),\cdots,x^*_m(k))
\end{pmatrix}
\end{array}
$$
\normalsize
The above iterative scheme involves finding an inverse of a tridiagonal matrix.
The inverse of the matrix can be computed by a simple algorithm presented in \cite{Lewis}.
We next present a convergence theorem for Newton's method.
\begin{proposition} \rm{(Kantorovich's Theorem \cite[p. 244]{Hubbard})} 
Let $\ba_0$ be a point in $\mathbb{R}^K$, $U$ be an open neighbourhood of $\ba_0$ in $\mathbb{R}^K$
and $F: U \to \mathbb{R}^K$ be a differentiable mapping, with its derivative $[DF(\ba_0)]$ invertible. 
Define
$$\bh_0 = -[DF(\ba_0)]^{-1}F(\ba_0),\quad \ba_1 = \ba_0 + \bh_0 \mbox{ and } U_1 = B_{|\bh_0|}(\ba_1).$$
If $\overline{U_1} \subset U$ and the derivative $[DF(\bx)]$ satisfies the Lipschitz condition
$$\|DF(\bu_1) - DF(\bu_2)\| \leq M|\bu_1-\bu_2|$$
for all points $\bu_1$ and $\bu_2 \in \overline{U_1}$ and if the inequality
$$|F(\ba)| \|DF(\ba_0)^{-1}\|^2 M \leq \dfrac{1}{2}$$
is satisfied, then the equation $F(\bx) = \b0$ has a unique solution in the closed ball $\overline{U_1}$ and Newton's method with initial
guess $\ba_0$ converges to it.
\end{proposition}

\begin{remark}
If we set $\ba=(0,\cdots,0)^t$ as the initial guess then it can be shown that we can take (see \rm{\cite[p. 240]{Hubbard}})
$$
M^2=4\dsum_{i=2}^m \Big( \dfrac{\mu^c_i}{C_{i-1}C_i} \Big)^2.
$$
Moreover, we have
$$
|F(\ba)|^2=(\mu^{c2}_1+\la^{c2}).
$$
If condition (\ref{stablecont}) is satisfied, $DF(\ba)$ is strictly diagonally dominant.
Hence, we have (see \cite{Moraca})
$$
\| DF(\ba)^{-1} \|\leq \dfrac{\sqrt{m}}{\min\Big\{ \frac{\mu^c_1}{C_1}+\th^c_1+ \frac{\mu^c_2}{C_1}, \th^c_2 + \frac{\mu^c_3}{C_2}-\frac{\mu^c_2}{C_1}, \cdots,
\th^c_{m-1} + \frac{\mu^c_m}{C_{m-1}}-\frac{\mu^c_{m-1}}{C_{m-2}}, \th^c_m - \frac{\mu^c_m}{C_{m-1}}\Big\}}.
$$
Thus, by Proposition 7, a sufficient condition for Newton's method to be convergent with the initial guess $\ba$ is
$$
\dfrac{m^2(\mu^{c2}_1+\la^{c2})\sum_{i=2}^m \Big( \frac{\mu^c_i}{C_{i-1}C_i} \Big)^2 }{\Big[\min\Big\{ \frac{\mu^c_1}{C_1}+\th^c_1+ \frac{\mu^c_2}{C_1}, \th^c_2 + \frac{\mu^c_3}{C_2}-\frac{\mu^c_2}{C_1}, \cdots,
\th^c_{m-1} + \frac{\mu^c_m}{C_{m-1}}-\frac{\mu^c_{m-1}}{C_{m-2}}, \th^c_m - \frac{\mu^c_m}{C_{m-1}}\Big\}\Big]^4} \leq \dfrac{1}{16}.
$$
\end{remark}

\begin{figure} 
\centering 
\begin{tikzpicture} [domain=0:20,scale=0.5] 
\node[minimum size=5mm,draw,circle] (y11) at (0,0) {};
\node[minimum size=5mm,draw,circle] (y12) at (3,0) {};
\node[minimum size=5mm,draw,circle] (y13) at (6,0) {};
\node[minimum size=5mm,draw,circle] (y14) at (9,0) {};
\node[minimum size=5mm,draw,circle] (y15) at (12,0) {};

\draw[<->,very thick] (y11) -- (y12);
\draw[<->,very thick] (y12) -- (y13);
\draw[<->,very thick] (y13) -- (y14);
\draw[<->,very thick] (y14) -- (y15);
\draw[<->,very thick] (y11) edge [bend right=15] (y13);
\draw[<->,very thick] (y11) edge [bend right=18] (y14);
\draw[<->,very thick] (y11) edge [bend right=21] (y15);
\draw[<->,very thick] (y12) edge [bend right=15] (y14);
\draw[<->,very thick] (y12) edge [bend right=18] (y15);
\draw[<->,very thick] (y13) edge [bend right=15] (y15);

\draw (-1,-2) -- (13,-2) -- (13, 1) -- (-1,1) -- (-1,-2);

\node[minimum size=5mm,draw,circle] (y21) at (0,4) {};
\node[minimum size=5mm,draw,circle] (y22) at (3,4) {};
\node[minimum size=5mm,draw,circle] (y23) at (6,4) {};
\node[minimum size=5mm,draw,circle] (y24) at (9,4) {};
\node[minimum size=5mm,draw,circle] (y25) at (12,4) {};

\draw[<->,very thick] (y21) -- (y22);
\draw[<->,very thick] (y22) -- (y23);
\draw[<->,very thick] (y23) -- (y24);
\draw[<->,very thick] (y24) -- (y25);
\draw[<->,very thick] (y21) edge [bend right=15] (y23);
\draw[<->,very thick] (y21) edge [bend right=18] (y24);
\draw[<->,very thick] (y21) edge [bend right=21] (y25);
\draw[<->,very thick] (y22) edge [bend right=15] (y24);
\draw[<->,very thick] (y22) edge [bend right=18] (y25);
\draw[<->,very thick] (y23) edge [bend right=15] (y25);

\draw (-1,2) -- (13,2) -- (13, 5) -- (-1,5) -- (-1,2);

\node[minimum size=5mm,draw,circle] (y31) at (0,8) {};
\node[minimum size=5mm,draw,circle] (y32) at (3,8) {};
\node[minimum size=5mm,draw,circle] (y33) at (6,8) {};
\node[minimum size=5mm,draw,circle] (y34) at (9,8) {};
\node[minimum size=5mm,draw,circle] (y35) at (12,8) {};

\draw[<->,very thick] (y31) -- (y32);
\draw[<->,very thick] (y32) -- (y33);
\draw[<->,very thick] (y33) -- (y34);
\draw[<->,very thick] (y34) -- (y35);
\draw[<->,very thick] (y31) edge [bend right=15] (y33);
\draw[<->,very thick] (y31) edge [bend right=18] (y34);
\draw[<->,very thick] (y31) edge [bend right=21] (y35);
\draw[<->,very thick] (y32) edge [bend right=15] (y34);
\draw[<->,very thick] (y32) edge [bend right=18] (y35);
\draw[<->,very thick] (y33) edge [bend right=15] (y35);

\draw (-1,6) -- (13,6) -- (13, 9) -- (-1,9) -- (-1,6);

\node[minimum size=5mm,draw,circle] (y41) at (0,12) {};
\node[minimum size=5mm,draw,circle] (y42) at (3,12) {};
\node[minimum size=5mm,draw,circle] (y43) at (6,12) {};
\node[minimum size=5mm,draw,circle] (y44) at (9,12) {};
\node[minimum size=5mm,draw,circle] (y45) at (12,12) {};

\draw[<->,very thick] (y41) -- (y42);
\draw[<->,very thick] (y42) -- (y43);
\draw[<->,very thick] (y43) -- (y44);
\draw[<->,very thick] (y44) -- (y45);
\draw[<->,very thick] (y41) edge [bend right=15] (y43);
\draw[<->,very thick] (y41) edge [bend right=18] (y44);
\draw[<->,very thick] (y41) edge [bend right=21] (y45);
\draw[<->,very thick] (y42) edge [bend right=15] (y44);
\draw[<->,very thick] (y42) edge [bend right=18] (y45);
\draw[<->,very thick] (y43) edge [bend right=15] (y45);

\draw (-1,10) -- (13,10) -- (13, 13) -- (-1,13) -- (-1,10);

\draw[->] (6,-2) -- (6, -3);
\draw[->] (6,2) -- (6, 1);
\draw[->] (6,6) -- (6, 5);
\draw[->] (6,10) -- (6, 9);
\draw[->] (6,14) -- (6,13);

\draw[dashed,->] (13,-0.5) -- (14.5, -0.5);
\draw[dashed,->] (13,3.5) -- (14.5, 3.5);
\draw[dashed,->] (13,7.5) -- (14.5, 7.5);
\draw[dashed,->] (13,11.5) -- (14.5, 11.5);

\end{tikzpicture} 
\caption{A 4 echelon inventory system.}
\end{figure}
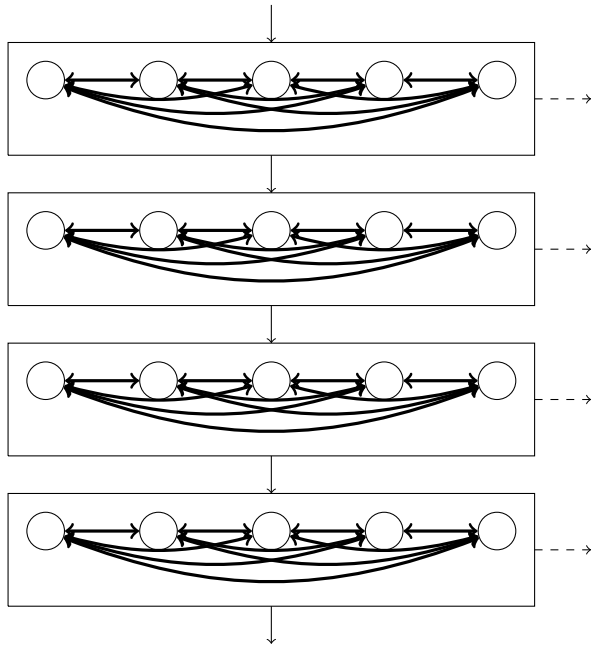 

\subsection{Illustrative example}
In this subsection, we illustrate the use of the proposed model by the following example.
Suppose that there are $4$ echelons in an inventory system and each echelon consists of $5$ warehouses, see Figure 5.
Within each echelon, there are lateral transshipment between the $5$ warehouses.
We assume all lateral transshipment rates are $\ga_{ij}=1$ and $L_i=20$ for all warehouses.
Within each echelon, the maximum supply rates and demand rates are the same.
In each echelon, the deterioration percentage in warehouse $i$ is given by
$$
\th_i=i \times 0.05 \quad \mbox{for } i=1,2,\cdots,5.
$$
Suppose that
$$
\mu^c_1=50,\ \ \mu^c_2=45,\ \ \mu^c_3=40,\ \ \mu^c_4=30 \mbox{ and } \la^c=5,
$$

Suppose we would like to find the inventory level in equilibrium for the 4th warehouse in the 3rd echelon.
If the whole inventory system is modelled then there are $4 \times 5=20$ states of inventory levels to be handled.
In what follows, we propose a two-phase procedure for finding the inventory level of a warehouse in a particular echelon.
 
\begin{itemize}
	\item Phase 1: We first aggregate the warehouses in each echelon to one aggregated warehouse, which means 
	$$C_i=5 \times 20 = 100 \quad \mbox{and} \quad\th^c_i=\dfrac{1}{5}(0.05+0.1+\cdots+0.25)=0.15.$$
	We then apply Newton's method to find the inventory level in equilibrium for each echelon. 
	The initial guess is set to be $\ba=(0,0,0,0)$ and the shopping criterion is the following
	$$
	\| F(x^*_1,\cdots,x^*_m) \|_2 \leq 10^{-10}.
	$$ 
	Newton's method converges in a few steps and the result is
	$$
	(x^*_1,x^*_2,x^*_3,x^*_4)= (53.0, 34.6, 24.9, 11.0).
	$$
	
	\item Phase 2: We focus on the 3rd echelon. The maximum supply rate at each warehouse depends on the inventory levels in the 2nd echelon.
	Here we use the result from Phase 1 and set	
	$$
	\mu_i=8\times\dfrac{34.6/5}{20}=2.77
	$$
	by using the average inventory level in the 2nd echelon.
	The demand rate at each warehouse depends on the inventory levels in the 3rd and 4th echelons.
  Here we use the result from Phase 1 and set	
  $$
	\la_i=6\times\dfrac{24.9/5}{20}\times\dfrac{20-11.0/5}{20}=1.33
	$$
	by using the average inventory level in the 3rd and 4th echelon.
	
	Solving the equilibrium point by the method in Section 2.2, we obtain
	$$
	(y^*_1,y^*_2,y^*_3,y^*_4,y^*_5)=(6.2,    5.6,    5.1,   4.6,    4.3).
	$$
Hence,	the 4th warehouse in the 3rd echelon has equilibrium inventory level equals $4.6$.
\end{itemize}
By the above procedure, we only need to handle $4+5=9$ states of inventory levels.

\subsection{Numerical examples}
In this subsection, we present some numerical examples following the illustrative example in Section 3.1.
In all numerical tests, we assume $C_i = 100$ for all echelons and within each echelon, the maximum inventory levels,
maximum supply rates and demand rates are the same.

For the numerical tests in Tables 6-8, we consider $m=2,4,8$ echelons in the system and each echelon consists of $n=2,4,8$ warehouses.
In each table, we give the maximum supply rate for each echelon.
In all cases the lowest echelon is subject to a total demand $\la^c=5$.
For each echelon the deterioration percentage in warehouse $i$ is given by
$$
\th_i=i \times 0.04 \quad \mbox{for } i=1,2,\cdots,n
$$
and all lateral transshipment rates are $\ga_{ij} = 1$ for $i \neq j$.
In each case, we first give the total inventory levels in equilibrium  at each echelon in the first column.
For each echelon, we then give the inventory levels in equilibrium  at each warehouses. 
We observe that when $m$ increases, the inventory levels at the lowest echelon in equilibrium are decreased and the changes become more significant when $n$ increases.
For each fixed $m$, when we increase the number of warehouses in each echelon, the total inventory levels in equilibrium  in each echelon decrease. This suggests that to reduce the inventory cost, one may consider to build more warehouses in each echelon.

\begin{table} \centering{\scriptsize
\begin{tabular}{ll|ll|ll}
\hline
\multicolumn{6}{c}{$m=2,\quad (\mu^c_1,\mu^c_2)=(50,30)$}\\
\hline
$n=2$&&$n=4$&&$n=8$\\
\hline
73.8	&(38.2	35.7)	&66.7&	(18.1 	17.1 	16.3  15.5)&	57.0	&(8.0 	7.8  7.5 	7.3  7.1 	6.9  6.7 	6.5)\\
60.9	&(32.5	28.7)	&50.0&	(14.4 	13.1  12.0  11.1)&	34.5	&(5.1 	4.9  4.7 	4.5  4.3 	4.1  4.0 	3.8)\\
\hline
\end{tabular}
\caption{The equilibrium points of the inventory level when $m=2$.}

\begin{tabular}{ll|ll|ll}
\hline
\multicolumn{6}{c}{$m=4,\quad (\mu^c_1,\cdots,\mu^c_4)=(90,70,50,30)$}\\
\hline
$n=2$&&$n=4$&&$n=8$\\
\hline
76.6&	(39.1	37.6)&	69.9&	(18.4	17.8	17.2	16.6)	&61.8&	(8.4	8.2	8.0	7.8	7.6	7.5	7.3	7.2)\\
69.4&	(35.8	33.6)&	58.8&	(16.0	15.1	14.3	13.6)	&46.3&	(6.5	6.3	6.1	5.9	5.7	5.5	5.4	5.2)\\
64.6&	(33.8	30.9)&	51.5&	(14.5	13.4	12.5	11.7)	&35.5&	(5.2	5.0	4.8	4.6	4.4	4.2	4.1	4.0)\\
56.7&	(30.4	26.5)&	41.1&	(12.1	10.9	9.9	  9.0)	  &19.7&	(3.0	2.8	2.7	2.6	2.5	2.4	2.3	2.2)\\
\hline
\end{tabular}
\caption{The equilibrium points of the inventory level when $m=4$.}

\begin{tabular}{ll|ll|ll}
\hline
\multicolumn{6}{c}{$m=8,\quad (\mu^c_1,\cdots,\mu^c_8)=(170,150,130,110,90,70,50,30)$}\\
\hline
$n=2$&&$n=4$&&$n=8$\\
\hline
78.0&	(39.4	38.6)&	71.8&	(18.5	18.1	17.8	17.4)	&65.4&	(8.7	8.5	8.4	8.2	8.1	8.0	7.9	7.7)\\
72.0&	(36.6	35.5)&	62.1&	(16.3	15.8	15.3	14.9)	&52.0&	(7.1	6.9	6.7	6.6	6.4	6.3	6.2	6.0)\\
69.7&	(35.5	34.2)&	57.2&	(15.1	14.6	14.0	13.5)	&44.1&	(6.1	5.9	5.8	5.6	5.5	5.3	5.2	5.1)\\
68.4&	(35.0	33.4)&	54.1&	(14.5	13.8	13.2	12.7)	&38.6&	(5.4	5.2	5.1	4.9	4.8	4.6	4.5	4.4)\\
67.3&	(34.6	32.7)&	51.8&	(14.1	13.3	12.6	12.0)	&34.2&	(4.9	4.7	4.5	4.4	4.2	4.1	4.0	3.8)\\
65.8&	(34.1	31.8)&	49.6&	(13.7	12.8	12.0	11.3)	&30.1&	(4.4	4.2	4.0	3.9	3.7	3.6	3.5	3.4)\\
63.1&	(33.1	30.2)&	46.2&	(13.1	12.1	11.2	10.4)	&25.1&	(3.7	3.6	3.4	3.3	3.1	3.0	2.9	2.8)\\
55.9&	(30.0	26.2)&	37.1&	(11.0	9.9	8.9	8.1)	&9.9&	(1.5	1.4	1.4	1.3	1.2	1.2	1.1	1.1)\\
\hline
\end{tabular}
\caption{The equilibrium points of the inventory level when $m=8$.}
}\end{table}

In Tables 9-11, we assume $m=4$ and $n=4$ with $\mu^c_1=90, \mu^c_2=70, \mu^c_3=50, \mu^c_4=30$ and $\la^c=5$.
We calculate the equilibrium points for the following nine cases where in each echelon
$$\th_i=i \times 0.02, \quad \th_i=i \times 0.04 \quad \th_i=i \times 0.08,
$$
and all lateral transshipment rates are 
$\ga_{ij} = 0.5,1,2$ for $i \neq j$.
We observe that the total inventory levels in equilibrium  at lower echelons are more sensitive to the change of $\th_i$. 
For each fixed $\th_i$, when $\ga_{ij}$ is increased, the less variation of the inventory in equilibrium among the warehouses is observed. The reason for this is when $\ga_{ij}$ is increased, the inventory sharing between warehouses are more active.

\begin{table} \centering{\scriptsize
\begin{tabular}{l|lll}
\hline
&$\ga_{ij} = 0.5$ &$\ga_{ij} = 1$ &$\ga_{ij} = 2$  \\
\hline
78.8&	(20.3	19.9	19.5	19.2)&	(20.3	19.9	19.5	19.2)&	(20.2	19.9	19.6	19.3)\\
72.6&	(19.0	18.4	17.9	17.4)&	(18.9	18.4	17.9	17.5)&	(18.8	18.4	18.0	17.6)\\
68.4&	(18.2	17.5	16.8	16.1)&	(18.1	17.4	16.8	16.3)&	(17.9	17.4	16.9	16.4)\\
60.8&	(16.7	15.7	14.8	14.0)&	(16.4	15.6	14.9	14.2)&	(16.1	15.5	15.0	14.5)\\
\hline
\end{tabular}
\caption{The equilibrium points of the inventory level when $\th_i=i \times 0.02$.}

\begin{tabular}{l|lll}
\hline
&$\ga_{ij} = 0.5$ &$\ga_{ij} = 1$ &$\ga_{ij} = 2$  \\
\hline
69.9&	(18.5	17.8	17.1	16.5)&	(18.4	17.8	17.2	16.6)&	(18.3	17.7	17.2	16.7)\\
58.8&	(16.2	15.2	14.3	13.5)&	(16.0	15.1	14.3	13.6)&	(15.8	15.1	14.4	13.8)\\
51.5&	(14.8	13.5	12.4	11.5)&	(14.5	13.4	12.5	11.7)&	(14.1	13.3	12.6	11.9)\\
41.1&	(12.6	11.0	9.7	8.8)&	  (12.1	10.9	9.9	9.0)   &	(11.6	10.7	10.0	9.4)\\
\hline
\end{tabular}
\caption{The equilibrium points of the inventory level when $\th_i=i \times 0.04$.}

\begin{tabular}{l|lll}
\hline
&$\ga_{ij} = 0.5$ &$\ga_{ij} = 1$ &$\ga_{ij} = 2$  \\
\hline
60.3&	(16.8	15.6	14.6	13.7)&	(16.7	15.6	14.6	13.8)&	(16.5	15.5	14.7	13.9)\\
44.1&	(13.3	11.7	10.4	9.4)&	(13.1	11.7	10.5	9.6)&	(12.7	11.6	10.6	9.8)\\
32.8&	(10.9	9.0	7.6	6.7)&	(10.4	8.9	7.7	6.9)&	(9.9	8.8	7.9	7.1)\\
16.2&	(6.0	4.6	3.7	3.1)&	(5.6	4.5	3.8	3.3)&	(5.2	4.4	3.9	3.5)\\
\hline
\end{tabular}
\caption{The equilibrium points of the inventory level when $\th_i=i \times 0.08$.}
}\end{table}


\section{Concluding remarks}
In this paper, we propose a continuous time model for a multi-echelon inventory system with deteriorating items.
Lateral transshipment is allowed with rate depends on the inventory levels of the corresponding warehouses.
A fast procedure based on Newton's method is developed for finding the equilibrium points of the system.
Numerical results indicate that the method is efficient.

For future research, one may consider reverse logistics. 
Returned products are collected and stored at lower echelons and transported to upper echelons for rework.
Another direction is to consider minimization of operation costs of the
system by including costs associated with the normal delivery, lateral
transshipment and inventory costs at all the warehouses.

\end{document}